\documentclass[10pt, reqno]{article}
\usepackage{hyperref}
\hypersetup{
    colorlinks=true,
    linkcolor=blue,
    citecolor=red,
    urlcolor=blue,
    pdfborder={0 0 0}
} 
\usepackage{fullpage}
\usepackage{enumerate}

\usepackage{amsmath,amsfonts,amssymb,graphics,amsthm, comment}

\usepackage{graphicx}
\usepackage{xcolor}
\usepackage{bbm}
\usepackage{wrapfig} 

\usepackage[font=sf, labelfont={sf,bf}, margin=1cm]{caption}

\usepackage{cleveref}
  \crefname{theorem}{Theorem}{Theorems}
  \crefname{thm}{Theorem}{Theorems}
  \crefname{lemma}{Lemma}{Lemmas}
  \crefname{lem}{Lemma}{Lemmas}
  \crefname{remark}{Remark}{Remarks}
  \crefname{prop}{Proposition}{Propositions}
\crefname{notation}{Notation}{Notations}
\crefname{claim}{Claim}{Claims}
  \crefname{defn}{Definition}{Definitions}
  \crefname{corollary}{Corollary}{Corollaries}
  \crefname{section}{Section}{Sections}
  \crefname{figure}{Figure}{Figures}
    \crefname{assumption}{Assumption}{Assumptions}

\newtheorem{thm}{Theorem}[section]

\newtheorem{lemma}[thm]{Lemma}
\newtheorem{corollary}[thm]{Corollary}
\newtheorem{prop}[thm]{Proposition}
\newtheorem{defn}[thm]{Definition}

\numberwithin{equation}{section}

\theoremstyle{definition}
\newtheorem{remark}[thm]{Remark}

\def\cP{\mathcal{P}}

\def\cH{\mathcal{H}}

\def \ve {\varepsilon}

\def \Gd {G^{\#\delta}}

\def\P{\mathbb{P}}
\def\E{\mathbb{E}}
\def\C{\mathbb{C}}
\def\R{\mathbb{R}}
\def\Z{\mathbb{Z}}
\def\N{\mathbb{N}}
\def\D{\mathbb{D}}

\def  \p- {p\textunderscore}

\def\eps{\varepsilon}

\def\ph{\varphi}

\def \d {{\delta}}

\DeclareMathOperator{\var}{Var}
\DeclareMathOperator{\Arg}{Arg}

\renewcommand{\i}{\text{\textnormal{int}}}

\newcommand{\abs}[1]{ \lvert #1 \rvert}

\title{A note on dimers and T-graphs}

\author{Nathana\"el Berestycki\thanks{Supported in part by EPSRC grants 
EP/L018896/1 and EP/I03372X/1} \and Benoit Laslier\thanks{Supported in part by 
EPSRC grant EP/I03372X/1} \and Gourab Ray\thanks{Supported in part by EPSRC 
grant EP/I03372X/1}}

\begin{document}

\maketitle
\begin{abstract}
The purpose of this note is to give a succinct summary of some basic properties of T-graphs which arise in the study of the dimer model. We focus in particular on the relation between the dimer model on the heaxgonal lattice with a given slope, and the behaviour of the uniform spanning tree on the associated T-graph. 
Together with the main result of the companion paper \cite{BLR16}, the results here show Gaussian free field fluctuations for the height function in some dimer models.
%
\end{abstract}



%
%
%
%

\section{Introduction}

The dimer model is a classical model from statistical mechanics which can be thought of 
as defining a natural random surface through its height function. 
We refer to \cite{KenyonSurvey} for a detailed survey on the subject. A key tool in the early study of the dimer model is Temperley's bijection, which relates the dimer model on domains satisfying certain boundary conditions to uniform spanning trees on a modified graph. This correspondence was extended first by Kenyon, Propp and Wilson \cite{KPWtemperley}  to so-called Temperleyan graphs. A more delicate but also more general extension was discovered by Kenyon and Sheffield \cite{dimer_tree}. This can be used in particular to relate the dimer model on (subsets of) the hexagonal lattice to spanning trees on modified graphs called T-graphs. This correspondence is the main concern of this paper. Our goal is to make precise a number of facts which are probably known in the folklore but for which we could not find a reference, as well as to prove a number of new estimates which are needed for the study of the dimer model on the hexagonal lattice. 

Our results can be summarised as follows:
\begin{enumerate}

\item We show that the height function of a dimer configuration is given by the winding of branches in the associated spanning tree on the T-graph (\cref{lem:heightwinding}).

\item We study the large-scale behaviour of simple random walk on T-graphs. In particular, we prove recurrence (\cref{T:recurrence}) as well as a uniform crossing estimate (\cref{prop:satisfy_crossing}).

\item\label{it:planarity} We construct discrete domains $U^\delta$ on the T-graph associated with any given continuum domain $D$ on the plane with locally connected boundary (\cref{prop:domain}). More precisely, given a plane $P \subset \R^3$ and a domain $D \subset P$ with locally connected boundary, we construct domains $U^\delta$ on an associated T-graph such that the height function of any dimer configuration on $U^\delta$ satisfies $h^\d (\partial U^\d) \to \partial D$ as closed sets in $\R^3$ when $\delta \to 0$. 

\item Finally, we prove in \cref{T:gibbs} that the correspondence between dimers and uniform spanning trees on finite portions of the hexagonal lattice considered by Kenyon and Sheffield \cite{dimer_tree} extend to the full plane. More precisely, for a given plane $P \subset \R^3$ as above, the local limits of uniform spanning trees on the associated T-graph, 
corresponds to whole plane dimer configurations on the hexagonal lattice whose law is the unique ergodic Gibbs measure on dimers with that slope constructed by Sheffield \cite{SSthesis}. Equivalently, the weak limit of uniform spanning trees on suitable large portions of T-graphs (which exists by a result in the companion paper \cite{BLR16}), is in `bijection' with the whole plane Gibbs measure on dimers of the corresponding slope.

\end{enumerate}

Together with the main result of \cite{BLR16} (Theorem 1.2 in that paper), this implies that the height function of dimer configurations on the hexagonal lattice with planar boundary conditions (see \cref{it:planarity} above), has fluctuations given in the scaling by the Gaussian free field, as stated precisely in Theorem 1.1 of \cite{BLR16}.

\section{Definition and construction}

\subsection{Lozenge tilings and height function}\label{sec:height_intro}

We will first give a short introduction to the lozenge tiling model to setup our 
notations and make the paper more self-contained. For more details see 
\cite{las2013lozenge,KenyonSurvey} for example.

We call \textbf{lozenge tiling} of a domain $D \subset \C$ a tiling of $D$ by 
tiles which are made by gluing two equilateral triangles with edge length 
$1$ (see 
\cref{fig:exemple_tiling}). Lozenge tilings are in bijection with perfect 
matchings of subsets of the hexagonal lattice and with surfaces in $\R^3$ formed 
by the boundary of stacks of unit cubes (step surfaces). These bijections should be visually 
evident from \cref{fig:exemple_tiling} and are standard so we will not give more 
details on them. Using these bijections, we will freely identify tilings, 
perfect matchings and step surfaces. When talking about perfect matchings, we 
will write $\cH$ for the infinite hexagonal lattice and $G$ for subgraphs of 
$\cH$. We will call one of the bipartite class of vertices in $\cH$ black and 
the other white and write $b$ and $w$ for vertices in these classes 
respectively.
\begin{wrapfigure}{R}{.4\textwidth}
\begin{center}
\includegraphics[width = 0.3 \textwidth]{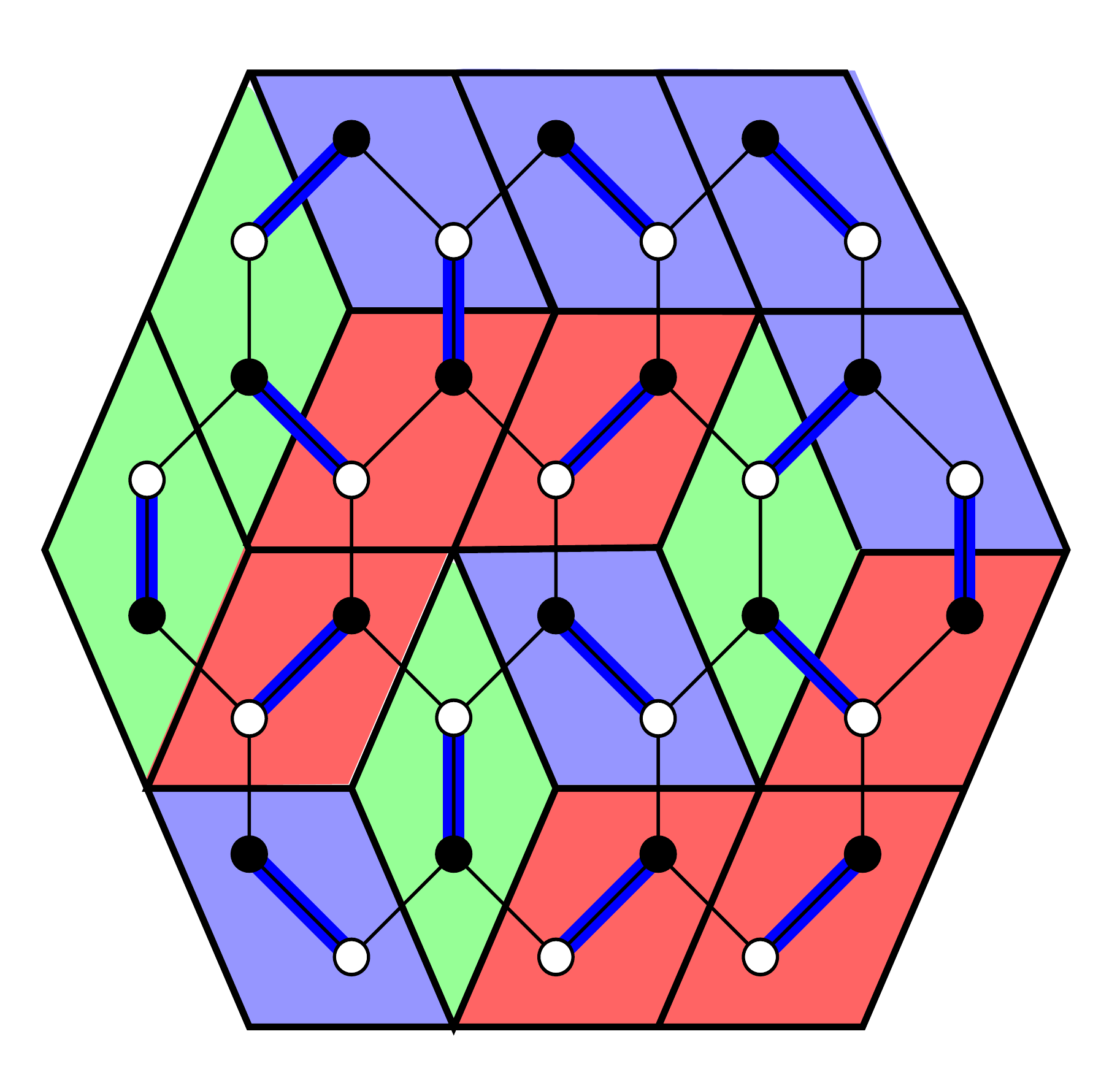}
\end{center}
\caption{Lozenge tiling of a domain. The colouring of the lozenges make the 
bijection with stepped surfaces apparent. The height function indicated is the 
$z$ coordinate of the corresponding point in $\Z^3$.}\label{fig:exemple_tiling}
\vspace{-1cm}
\end{wrapfigure}


%

When $D$ is bounded, we are interested in the uniform measure on lozenge tiling
of $D$. When $D$ is not bounded, one cannot simply take the uniform measure but 
the case $D = \C$ is well understood via the following theorem.
\begin{thm}\cite{SSthesis}\label{thm:uniqueness_dimers}
  For all $p_a,p_b,p_c$ in $(0,1)$ such that $p_a+p_b+p_c=1$, there
  exists a unique ergodic translation invariant Gibbs
  measure $\mu$ on lozenge tiling coverings of the plane such that
\begin{itemize}
\item Vertical (resp. north east- south west, north west-south east) lozenges appear with probability $p_a$ (resp. $p_b,p_c$).
  \item For any finite subgraph $G \subset \cH$, the measure on $G$,
    conditioned on the state of $\cH \setminus G$, is uniform over all
    possible dimer configurations in $G$.
  \end{itemize}
\end{thm}
The translation invariance implies directly that the expected height
change is linear. Correlations between tiles in these measures are known to be expressed in terms of determinants with an explicit (and quite simple) kernel (\cite{KenyonOkounkovSheffield}) but we will not use this fact here. 

\medskip
The height function with respect to a plane $\cP$ in $\R^3$ at some
vertex $x$ of the tiling is defined to be the distance of the
corresponding point of the surface to $\cP$ (up to a global scaling
factor). For example \cref{fig:exemple_tiling} shows the height function with 
respect to the horizontal plane $z = 0$ (using the standard coordinates in 
$\R^3$).

The above definition can be made more algebraic by giving the following equivalent definition. Let $\phi_{\text{ref}}$ be a function on oriented edges such that $\phi_{\text{ref}} (bw) = - \phi_{\text{ref}}(wb)$ for all edges $bw$ and such that 
\begin{align*}
 \forall w,  \,\,\sum_{b \sim w} \phi_{\text{ref}}(wb) &= 1 &  \forall
 b,  \,\, \sum_{w \sim b} \phi_{\text{ref}}(bw) &=-1 .
\end{align*}
To make statement about planarity simpler, we will also always assume
that $\phi_{\text{ref}}$ is periodic even though it is not strictly
required for this construction to hold. This function
$\phi_{\text{ref}}$ is called the reference flow and in the previous
geometric definition each reference plane corresponds to a particular
choice of $\phi_{\text{ref}}$.  Given a tiling, let $M$ be the set of
edges of $\cH$ that correspond to a lozenge (the blue edges in
\cref{fig:exemple_tiling}). We define a flow $\phi_M$ on oriented
edges by $\phi_M(wb) = \mathbbm{1}_{(wb) \in M}$ and $\phi_M(bw) =
-\mathbbm{ 1}_{(wb) \in M}$. We now define the dual flow $(\phi_M -
\phi_{\text{ref}})^\dagger$ on edges of the dual of $\cH$ (which is a
triangular lattice). The flow $(\phi_M -
\phi_{\text{ref}})^\dagger$ on an oriented 
dual edge $e^\dagger$ crossing an oriented edge $e$ in the primal (so that 
$e^\dagger$ is just a rotation of $e$ by $\pi/2$)  is the same as that in $e$. 
Since $(\phi_M - \phi_{\text{ref}})$ is divergence free by definition, we see 
that $(\phi_M - \phi_{\text{ref}})^\dagger$ is a gradient flow so it has a 
primitive $h_M$ and we define $h_M$ to be the height function of $M$ with 
reference flow $\phi_{\text{ref}}$. The global additive constant is fixed 
arbitrarily. A convenient choice when we work in a bounded domain is to set 
$h=0$ on some fixed boundary point.


It is easy to see that the height function determines the tiling
(assuming the reference is known). One can also check (it is quite
clear from \cref{fig:exemple_tiling}) that in a bounded domain $D$,
the height function $h_M$ along the boundary of $D$ is independent of
$M$ (as long as $M$ is a tiling of $D$). We can therefore talk of the
boundary height of $D$ without specifying a tiling inside. We say that a 
domain has \textbf{planar boundary with width $C$} if there exists a linear 
function $f$ such that $\abs{f-h} \leq C$ on $\partial D$.

\subsection{T-graph construction}\label{sub:construction}

%
In this section we construct the T-graphs and state some of its geometric 
properties that will be needed later.

\begin{wrapfigure}{R}{6.5cm}
\begin{center}
  \includegraphics[width=5cm]{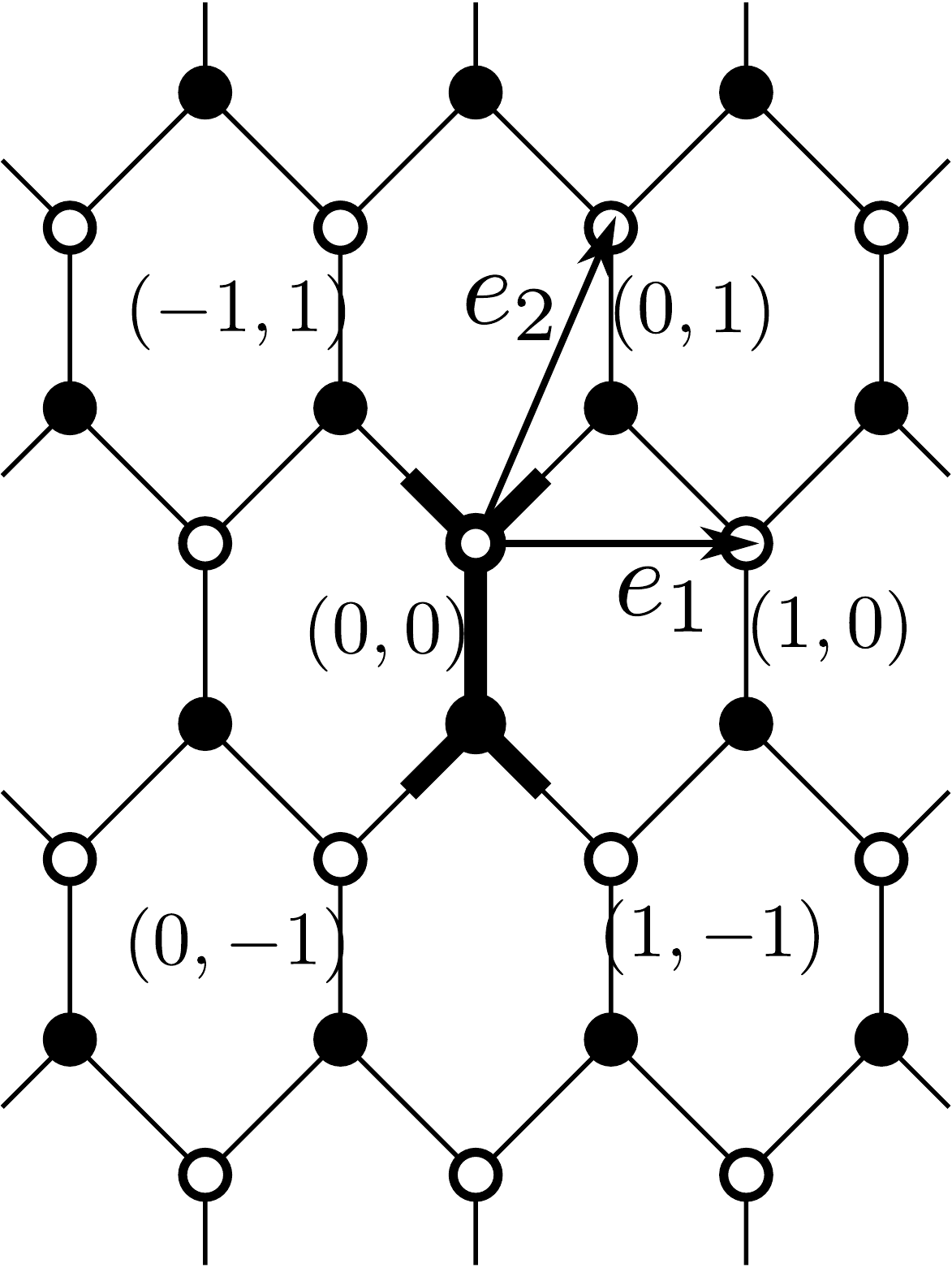}
\captionsetup{width=6cm}
\caption{An illustration of the coordinates we use on the hexagonal lattice. Near each vertical edge are indicated the (common) coordinates of its two endpoints.}
\label{fig:coordonees}
\end{center}
\vspace{-1cm}
\end{wrapfigure} 

We start by defining suitable coordinates on the infinite hexagonal
lattice $\cH$. 
We embed the hexagonal lattice $\cH$ in the plane in the way
  represented in Figure \ref{fig:coordonees}. The figure which is drawn with 
thicker lines in \cref{fig:coordonees} is
  called the fundamental domain and we write it as $\cH_1$.\ We let $e_1$ and 
$e_2$
  be the two vectors represented in \cref{fig:coordonees}. Given $v$ a vertex of $\cH$, we call
  coordinates of $v$ the unique $(m,n)\in \Z^2$ such that $v-me_1-ne_2 \in
  \cH_1$. Note that given $(m,n)$ there are exactly two vertices
  with coordinates $(m,n)$, the top one is a white vertex by convention and 
hence the
  bottom one is black. We will write $m(v)$ and $n(v)$ for the coordinates 
of the vertex
  $v$. We will also write $b(m,n)$ and $w(m,n)$ for the black and
  white vertices with coordinates $(m,n)$.
  
   We write $\cH^\dagger$ for the dual graph of $\cH$. This is a triangular 
lattice. Each of its faces contains a vertex of $\cH$ and a face is 
called black or white according to the colour of that vertex. Vertices of $\cH^\dagger$ 
can be associated to the point in the centre of a face of $\cH$.
For a vertex $v$ of
  $\cH^\dagger$ we let $(m(v),n(v))$ be the (common) coordinates of the two
  vertices of the face corresponding to $v$ located just to the right of $v$.
  
 T-graphs will be defined as the primitive of a specific gradient flow on $\cH^\dagger$ which we define now.
Let us fix $p_a, p_b, p_c$ in $(0,1)$ such that $p_a + p_b + p_c =1$. Let $\Delta$ be a triangle in the complex plane with angles $\pi p_a, \pi p_b, \pi p_c$. We write $A$, $B$, $C$ its vertices and $\alpha = \vec{BC}$, $\beta = \vec{CA}$, $\gamma = \vec{AB}$ seen as complex numbers. We have $\alpha + \beta + \gamma = 0$. Let $\lambda$ be a complex number of modulus one. We will require later $\lambda$ to be outside a set of Lebesgue measure $0$ and all references to almost every $\lambda$ are with respect to the Lebesgue measure on the circle.
  
  We can now define the following flow on oriented edges between a white vertex $w$ and a black vertex $b$:
  \[
  \phi(w b) = \Re\left( \lambda^{-1} (\tfrac{\beta}{\gamma})^{-m(w)} (\tfrac{\beta}{\alpha})^{-n(w)}  \right) \alpha \lambda (\tfrac{\beta}{\gamma})^{m(b)} (\tfrac{\beta}{\alpha})^{n(b)}, 
  \]
and $\phi(b w) = - \phi(w b)$. We then define the dual flow $\phi^\dagger$ by 
rotating $\phi$ anticlockwise by $\frac \pi 2$, i.e crossing the edge $wb$ with 
the white vertex on the left gives flow $+ \phi(w b)$.

One can check from the formula that the circulation of the dual flow $\phi^\dagger$ around
any triangle (sum of the flow along the edges in a triangular face) is $0$ so $\phi^\dagger$ is a gradient flow. Hence there
exists a function $\psi$ on the vertices of the triangular lattice $\cH^\dagger$ (we write $\psi_{\lambda \Delta}$ if we want
to emphasize the dependence), unique up to an additive constant such
that for all adjacent vertices of the triangular lattice $v, v' \in \cH^\dagger$,  $\psi(v')-\psi(v) =
\phi^\dagger(vv')$. We fix the additive constant by specifying $\psi$
to be $0$ on the dual vertex just to the left of the fundamental domain 
$\cH_1$. We can extend $\psi$ affinely to the edges of
$\cH^\dagger$ so that $\psi$ maps $\cH^\dagger$ to a connected union
of segments in $\C$ and we write $T_{\lambda \Delta} = \psi_{\lambda
  \Delta}(\cH^\dagger)$ (see \cref{fig:image_locale} for an
example). We call the image of $\cH^\dagger$ under $\psi$ the
\textbf{T-graph} with parameters $\Delta$ and $\lambda$.




\begin{figure}
\centering
\includegraphics[width=.3 \textwidth]{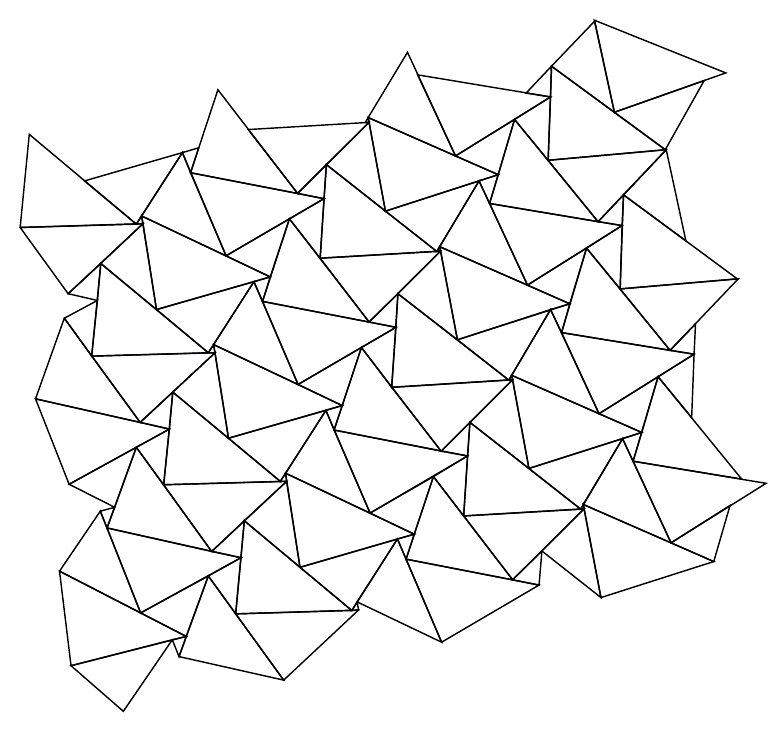}\quad \quad 
  \includegraphics[width = 0.5 \textwidth]{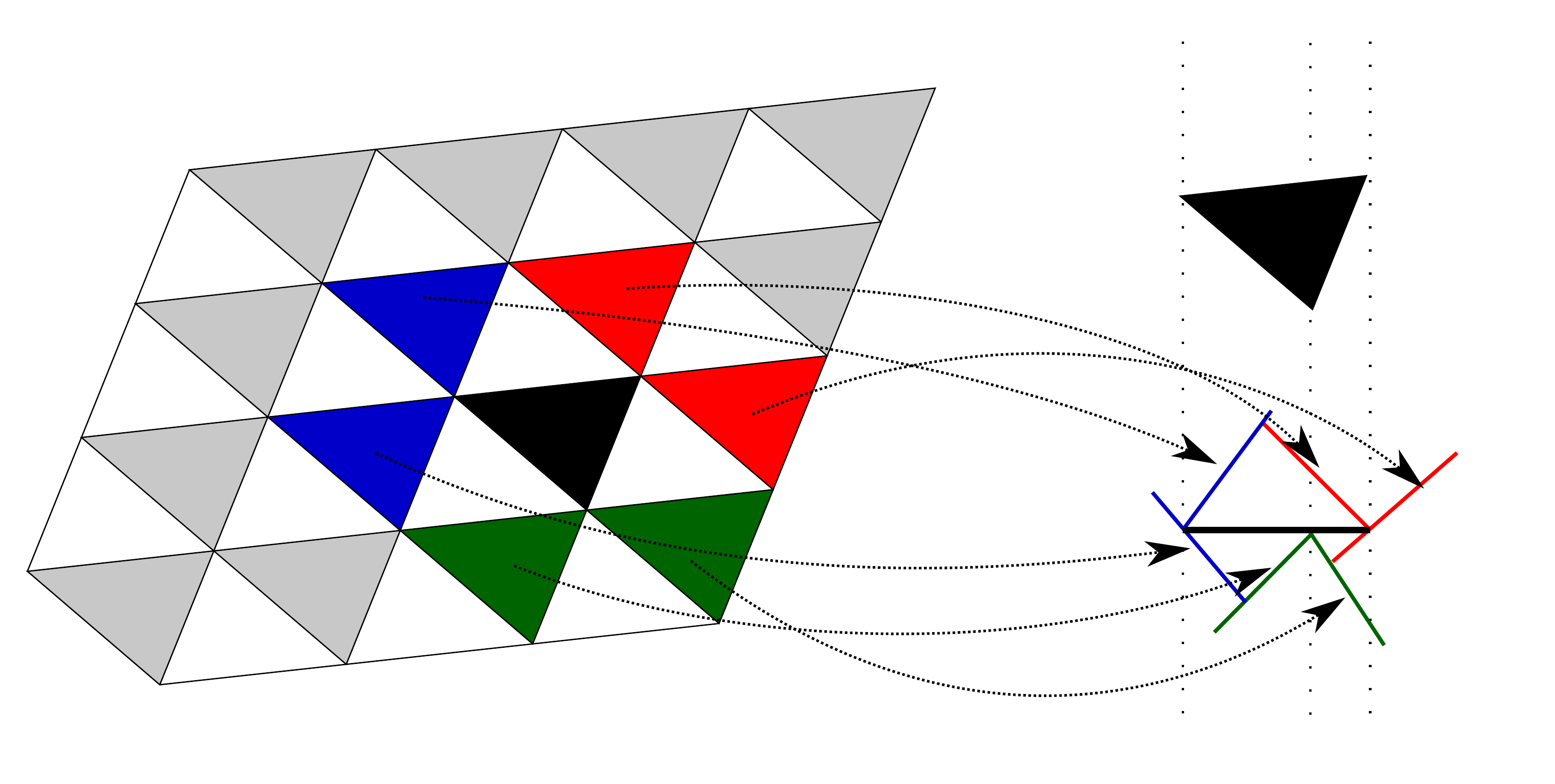} 

\caption{Left: A part of a T-graph. Right: A schematic view of the image of a black face and its 
neighbours in $\cH^\dagger$ (for clarity black faces are colored differently). Each black face is projected into a segment. The vertex of the triangle which gets projected in the interior of the segment becomes a vertex of the T-graph where two other segments meet.}
\label{fig:image_locale}
\end{figure}


\begin{remark}
One can possibly get some intuition about the construction by comparing with what we get if we omit the real part in the definition of $\phi$ above.\ In that case all faces would have been mapped to similar triangles and we would actually have defined simply the embedding of the triangular lattice with triangles of angles $\pi p_a$, $\pi p_b$, $\pi p_c$. The T-graph can be seen as a perturbation of that embedding where triangles corresponding to black faces are flattened into segments but that otherwise preserves the adjacency structure. We will also see in \cref{prop:translation} point (i) that at large scale the two embedding are close\footnote{up to global factor $\frac 1 2 $}.
\end{remark}

\section{Properties of T-graphs}
 
\subsection{Geometric properties} 
 
The following proposition tells us how the value of $\lambda$ behave under translation.

\begin{prop}\label{prop:translation}
  Let $T_{\lambda\Delta}$ be the graph constructed above (recall that we
  choose $\psi=0$ on the dual vertex just to the left of the 
fundamental
  domain $\cH_1$). Let $v $ be 
the
 vertex of $\cH^\dagger$ with coordinates $(m,n)$ and let
  $T'$ be the graph constructed in the same way but taking
  $\psi(v)=0$. Then we have $T' = T_{\lambda'\Delta}$ with $\lambda'=
  \lambda (\frac{\beta}{\gamma})^{m} (\frac{\beta}{\alpha})^{n}$. In particular, if we pick $\lambda$ to be uniform in the unit circle, then $T_{\lambda\Delta}$ becomes translation invariant.
\end{prop}

Here are some other geometric facts about $T$, see for instance \cite{LaslierCLT} :
\begin{prop}\label{prop:prop_geometrique} 
  $T$ has the following properties (see Figure \ref{fig:image_locale}):
\begin{enumerate}[{(}i{)} ]
\item \label{prop:prop_geometrique:lineaire} Let $\ell(m,n)= \frac{\alpha}{2}m - \frac{ \gamma}{2} n$, then $\psi(v) - \ell(m(v),n(v))$ is bounded. Furthermore $\ell$ is invertible so any point $z \in \C$ is at bounded distance from $T$.
\item \label{prop:prop_geometrique:face_noire} The image of any black
  face of $\cH^\dagger$ is a segment.
\item \label{prop:prop_geometrique:face_blanche} The image of any
  white face is a contraction and rotation of $\Delta$. In particular, the map 
preserves orientation.

\item \label{prop:prop_geometrique:non_degenere} The length of
  segments are bounded above and below uniformly in
  $\lambda$ and for almost every $\lambda$ no triangle is degenerate to a point.
\item \label{prop:prop_geometrique:endpoints} For any $\lambda$ (resp. almost every $\lambda$),
    for any vertex $v$ of $\cH^\dagger$, $\psi(v)$ belongs to at least (resp. exactly) three
    segments. Generically any vertex is an endpoint of two segments
    and is in the interior of
    the third one. All endpoints of segments are of the above form
    $\psi(v)$ with $v$ a vertex of $\cH^\dagger$. We call these points \textbf{vertices of the T-graph}.
 \item \label{prop:prop_geometrique:non_recouvrement} The triangular images of
white faces
    cover the plane and do not intersect, that is, any $z \in \C$ not in a 
    segment belongs to a unique face of the T-graph.
  \item \label{prop:prop_geometrique:non_intersection} If two segments intersect, the intersection point is an endpoint of at least one of the segments.
 \end{enumerate}
\end{prop}
\textbf{From now on we assume that $\lambda$ is chosen so that no triangle is
degenerate and all vertices belong to exactly three segments.} We say
that the corresponding T-graph is \textbf{non-degenerate}.


%
%
We can also naturally define \textbf{edges of the T-graph} to be the portion of the segments joining two vertices and faces to be the connected components of the complement of the segments.
From the \cref{prop:prop_geometrique}, we see the following correspondences. 
\begin{center}
\begin{tabular}{|c|c|c|}
 \hline 
 T-graph & Hexagonal lattice & Triangular lattice \\ 
 \hline 
 Segment & Black vertex & Black face \\ 
 \hline 
 Face & White vertex & White face \\ 
 \hline 
 Vertex & Face & Vertex \\ 
 \hline 
 Edge & & Edge\\
 \hline
 \end{tabular}
 \end{center}
 
Note that the first three correspondences are really bijective while in the last case, some edges of $\cH^\dagger$ do not correspond to any edge of $T$. 
 With an abuse of notation, we will allow ourself to write $\psi(w)$ and $\psi(b)$ for the corresponding triangle or segment. We will also write $\psi^{-1}$ for the inverse of this one to one correspondence. 

Thanks to the above abuse of notation, $\psi$ describes an embedding of $\cH$ in the plane:

\begin{corollary}\label{prop:planarity1}
The following embedding of $\cH$ in the plane is proper, i.e. has non crossing 
edges: each white vertex $w$ is placed in the centre of mass of the triangle 
$\psi(w)$, each black vertex $b$ is placed at the unique vertex of $T$ in the 
interior of $\psi(b)$. Draw edges between a white vertex $w$ and a black vertex $b$ if the segment corresponding to $b$ share more than one point with the face corresponding to $w$ in the T-graph.
\end{corollary}
\begin{proof}
Clearly a midpoint of a segment can be identified with a black face of the triangular lattice. We now simply observe that around each such vertex we drawn a hexagon:  

It is clear that we described an embedding of $\cH$ so we only have to check 
that edges are non crossing. For this, note that any edge $(bw)$ lies inside the 
triangle $\psi(w)$. Since triangles do not intersect, no pair of edges $(bw)$, 
$(b'w')$ with $w \ne w'$ intersects. Now for any $w$, its three adjacent edges 
are straight lines towards different points of $\psi(w)$ so they do not 
intersect. Further, the segment corresponding to a black vertex which is not a neighbour of the white vertex $w$ in the hexagonal lattice share at most one point with the face of $w$. The proof is now complete.
\end{proof}

\begin{figure}[h]
\begin{center}
\includegraphics[width = .3 \textwidth]{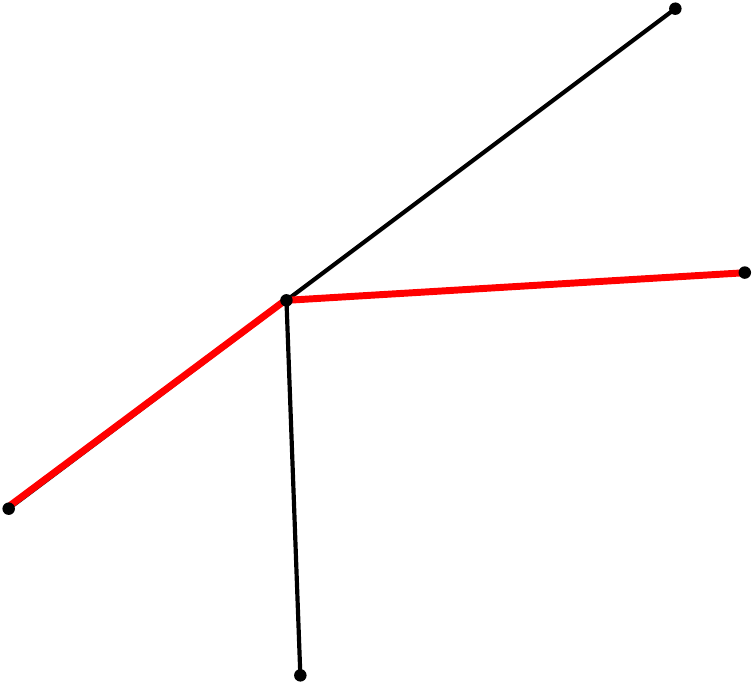} \includegraphics[width=.3\textwidth]{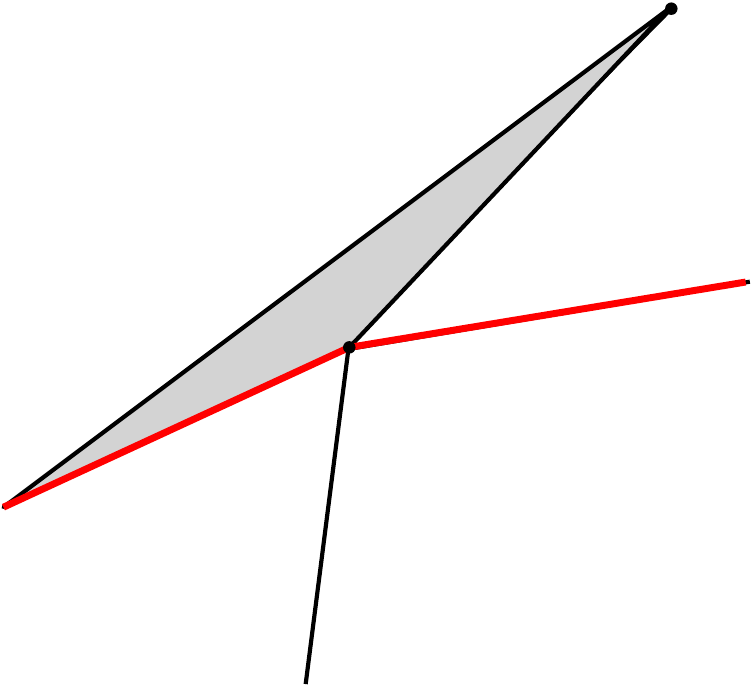}
\end{center}
\caption{Thickening of an edge into a black triangle, and effect on loop.}
\end{figure}

\begin{figure}[h]
\begin{center}
\includegraphics[width=.3\textwidth]{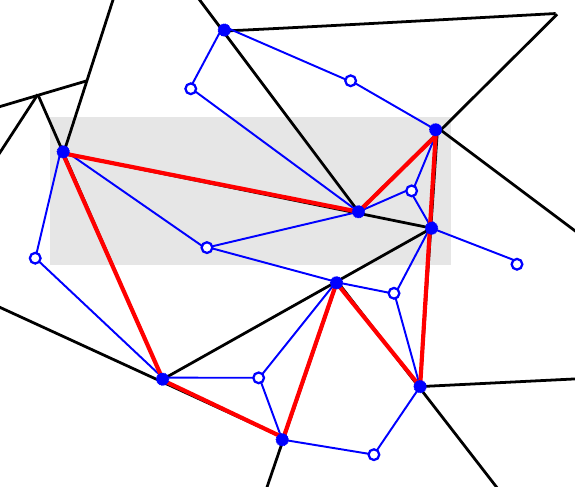} \hspace{.1\textwidth}
\includegraphics[width=.3\textwidth]{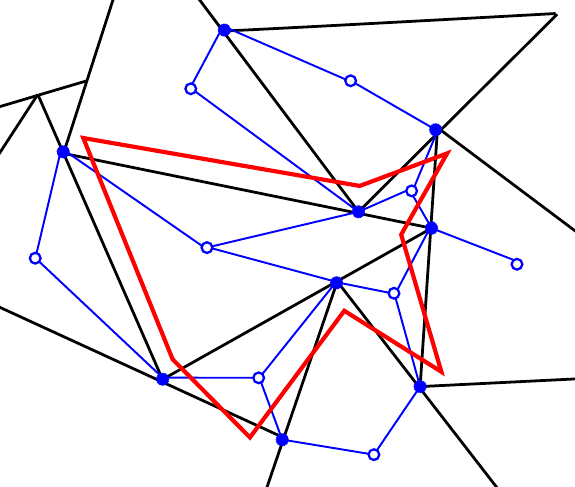} \\ \vspace{1cm}
\includegraphics[width=.3\textwidth]{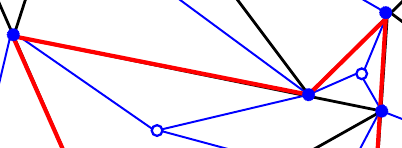}\hspace{.1\textwidth} 
\includegraphics[width=.3\textwidth]{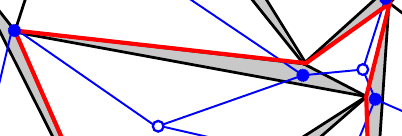} 
\end{center}
\caption{Top-Left: A portion of a T-graph with a loop $C$ in red along with the 
embedding 
of the hexagonal lattice determined by the T-graph as in 
\cref{prop:planarity1}. Top-Right: To visualise the path $\psi^
{-1}(C)$, one can shift every vertex of $C$ in the direction where $3$ 
triangles meet at the vertex (into the middle triangle, to be more precise). 
This path is drawn and can be readily seen as a dual path of the hexagonal 
lattice.
Bottom-Left :  The greyed area from the 
top-left figure zoomed in. Bottom-Right : Another way to visualise 
$\psi^{-1}(C)$: a perturbation of the T-graph where black faces are sent to 
thin triangles. The red path on that perturbation clearly coincide with the one 
in the top-left part.}
\label{fig:embedding}
\end{figure}

Finally one can associate sub-domains of a T-graph to sub domains of $\cH$ even though it require some care on the boundary.
\begin{prop}\label{prop:planarity2}
Let $C$ be a simple (unoriented) loop on a T-graph $T$, then $\psi^{-1}(C)$ is 
a simple close curve on $\cH^\dagger$ (each interior point actually 
has two pre-images, we use the one lying on the edge $e \in \cH^\dagger$ where $e$ is the edge which is
not mapped to the whole segment). Let $U$ be the subgraph of $\cH$ inside 
$\psi^{-1}(C)$ and let $D$ be the open domain strictly inside $C$. We have
\begin{itemize}
\item The white vertices of $U$ are exactly the $w$ such that $\psi(w) \subset \bar D$.
\item Any segment of $T$ which is strictly inside $C$ except possibly at finitely many points corresponds to a black vertex of $U$. All the neighbours of such a vertex are in $U$.
\item For a segment $S$ of $T$ with an infinite intersection with $C$, consider the side of $S$ where it has a single adjacent triangle. The black vertex $\psi^{-1}(S)$ is in $U$ if and only if that side of $S$ is inside $C$. These vertices have at least one neighbour in $U$ and one outside $U$.
\end{itemize}
\end{prop}


\begin{proof}

Firstly, it is easy to check that  $\psi^{-1}(C)$ is a simple loop on $\cH^\dagger$. Indeed with our convention, $\psi^{-1}$ is injective and edges including their endpoints in the T-graph are sent to edges including their endpoints in the triangular lattice.

We can clearly define a perturbation of $\psi$ such that black triangles are sent to thin triangles rather than line segments and such that it respects the structure of the T-graph otherwise. This perturbation defines a proper embedding of $\cH^\dagger$ and $\cH$, therefore we can identify $U$ by looking at the interior of $\psi^{-1}(C)$ in that embedding. In that embedding, the first two points are immediate. The last one is clear from the picture of the perturbed embedding in the bottom-right panel of \cref{fig:embedding}.
\end{proof}

\subsection{Uniform crossing estimate} \label{sec:unif_cross}

\begin{defn}
Let $T$ be a non-degenerate T-graph.\ The random walk $X_t$ on $T$ is the continuous time Markov process on the vertices of $T$ defined by the following jump rates. If the process is at a vertex $v$
of $T$, call $v^+, v^-$ the endpoints of the unique segment which
contains $v$ in its interior. The rates of the jumps from $v$ to $v^\pm$ are $1/|v^\pm-v|$.
\end{defn}
Note that this random walk 
is automatically
a martingale thanks to the choice of the jump rates. The jump rates
defined above allow us to consider a T-graph as a weighted oriented
graph. From now on we view a T-graph either as a weighted directed graph or
as a subset of $\C$ as required by context. We denote by $X^v$ the walk started 
in $v$.


%
%
%
%

We now prove the Russo-Seymour-Welsh type uniform crossing estimate. The key 
input is the following uniform ellipticity bound. 
\begin{lemma}[\cite{LaslierCLT} Proposition 2.22 ]\label{prop:ellipticite_uniforme}
  There exists $c_0 > 0$ (depending continuously on $\Delta$)
  such that for any ${\bf n}\in\mathbb S^1$ and any $v\in T$,
  \[
   c_0 <\var(X^v_1 \cdot{\bf n}) < c_0^{-1} .
  \]
\end{lemma}

Using this bound we first control the angle at which the random walk
exits a ball. Let $T^\d$ denote the T-graph rescaled by $\delta$. We
emphasise that the random walk keeps the original unrescaled time 
parametrisation.
\begin{lemma}\label{lem:eta}
There exists $\delta_1$ and $\eta$ such that for all $\delta \leq
\delta_1 $, for any $v \in T^{\d}$, if $\tau$ denotes the exit time of $B(v, 1)$ then for all $\theta \in [-\pi,\pi)$,
\[
\P\left[ \Arg(X^v_\tau - v) \in [\theta, \theta + \pi - \eta] \right] > \eta
\]
where $\Arg$ denote the usual argument in $[-\pi,\pi)$ but the
interval $[\theta, \theta + \pi - \eta]$ is interpreted cyclically
if $\theta + \pi - \eta \ge \pi$.
\end{lemma}
\begin{proof}
First let us prove that we can choose $\eta$ such that for all
$\theta \in [-\pi,\pi)$
\begin{equation}
\P\left[ \arg(X^v_\tau - v) \in [\theta, \theta + \pi - \eta] \right]
> 0. \label{eq:eta}
\end{equation}

Let us translate and rotate the coordinate frame to set $v =0$ and $\theta = \eta/2$ so that the arc $[\theta , \theta + \pi - \eta]$ is symmetric around the vertical axis. We let $Y_t$ denote the projection of the random walk on the vertical axis. 

Since the random walk is a martingale we have $\E[Y_\tau] = 0$ by
optional stopping (which we can use since $Y_{\tau \wedge t}$ is uniformly
bounded by $1$). Further
\cref{prop:ellipticite_uniforme} ensures that
$\var(Y_\tau) \geq cc_0$ where $c_0$ is as in 
\cref{prop:ellipticite_uniforme}. To see this first notice from
\cref{prop:ellipticite_uniforme} that $(Y_k^2 - c_0 \delta^2 k)_{k \in
  \N}$ is a submartingale. Further notice that by Burkholder--Davis--Gundy
inequality  with $p=2$ we have $\E(\tau) \ge c\delta^{-2}$. Using this, the optional stopping theorem and the monotone and dominated convergence theorems, 
\begin{equation}
  \label{eq:36}
  \E(Y_\tau^2) \ge c_0 \delta^2 \E(\tau) \ge cc_0.
\end{equation}
Suppose by contradiction that
\[
\P\left[ \arg(X^v_\tau - v) \in [\theta, \theta + \pi - \eta] \right] = 0.
\]
Then $Y_\tau$ is a random variable in $[-1+O(\delta),\sin(\eta/2)
+O(\delta)]$. Since $\E[Y_\tau] =0$, we have a bound $\var(Y_\tau)
\leq (\sin(\eta/2) +O(\delta)) (1+O(\delta))$. Indeed, this bound is
valid for any centred random variable in $[-1-O(\delta),\sin(\eta/2) +
O(\delta)]$. However we saw above that $\var(Y_\tau) \geq cc_0$.
This is a contradiction if $\eta,\delta$ are small enough. 

To obtain a lower bound of at least $\eta$ in \eqref{eq:eta}, just observe that there is also a universal bound of the variance of variables in $[-1,1]$ with $\P[Y > \sin(\eta/2)] \leq \eta$ and $\E[Y] = 0$ so the same proof applies.
\end{proof}

\begin{figure}
 \centering
\includegraphics[scale = 0.5]{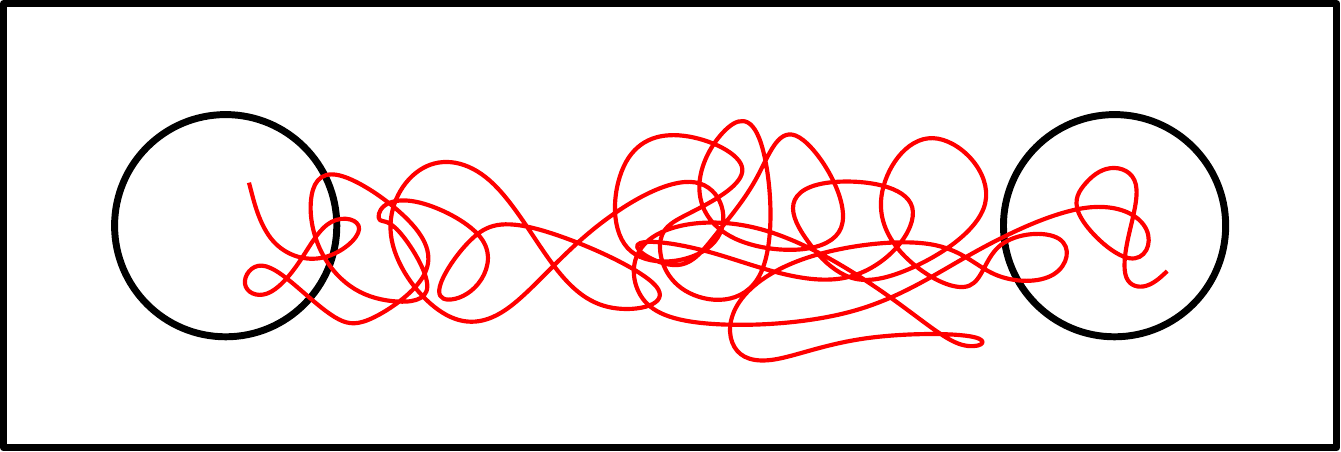}
\caption{An illustration of the uniform crossing 
property.}\label{fig:crossing}
\end{figure}

We now give a precise definition of what it means for a graph $G$ (embedded in $\C$) to satisfy the \textbf{uniform crossing 
property}. Let $G^\d$ denote the rescaling of the edges of the graph $G$ by $\delta$. Let $A^\d$ denote the graph spanned by the vertices of $G^\d $ in $A$ for any $A \subset \C$. 
Let $R$   
   be the horizontal rectangle $[0,3]\times [0,1]$ and $R'$ be the vertical 
rectangle $[0,1]\times [0,3]$. Let $B_1 := B((1/2,1/2),1/4)$ be the 
\emph{starting ball} and $B_2:= B((5/2,1/2),1/4) $ be the \emph{target ball} 
(see \cref{fig:crossing}). There 
exists universal constants $\delta_0 >0$ and $\alpha_0>0$ such that for all $z 
\in \C$, $\delta \in (0,\delta_0)$, $v \in B_1$ such that $v+z \in G^\d$,
\begin{equation}
  \P_{v+z}(X \text{ hits }B_2+z \text{ before exiting } (R+z)^\d) 
>\alpha_0.\label{eq:cross_left_right}
\end{equation}
The same statement as above holds for crossing from right to left, i.e., for 
any $v \in B_2$, \eqref{eq:cross_left_right} holds if we replace $B_2$ by 
$B_1$. 
Also, the same statement holds for the vertical rectangle $R'$. Let $B_2' = 
B((1/2,5/2),1/4)$.  
Then for all $z \in \C$, $\delta \in (0,\delta_0)$, $v \in 
B_1$ such that $ v+z \in \Gd$,
\[
 \P_{v+z}(X \text{ hits }B'_2+z \text{ before exiting } (R'+z)^\d)>\alpha_0.
\]
Again, the same statement holds for crossing from top to bottom, i.e., from 
$B_2'$ to $B_1$.

\begin{thm}\label{prop:satisfy_crossing}
Any non degenerate T-graph satisfies the uniform crossing property.
\end{thm}
\begin{proof}
We will only prove \eqref{eq:cross_left_right} as the proof of the other cases 
are identical.
The idea is to look at the random walk stopped when it exits small
macroscopic discs. We then construct a function on the rectangle which
has a positive probability to decrease at each of these steps and such
that if it decreases for the $k$ first steps (where $k$ is uniform in
the location of the rectangle) a crossing has to happen.


 Let $f(x, y) = -x + 10(y-\frac12)^2 +
[10(x-\frac 52) \vee 0 ]^{10}$. Notice that $f$ achieves its global
minimum in $B_2$ and its value on $\partial R$ is bigger than in
$B_1$. Furthermore,  it has bounded second derivatives in $R \setminus
B_2$ and the norm of its gradient is also lower bounded in $R
\setminus B_2$. Choose $\eta$ as in \cref{lem:eta}. There exist $r,\eta'>0$ such that
\[
\forall z \in R \setminus B_2,\, \, \exists \theta \, \, s.t \, \, \forall  \, 
\,\varphi \in [\theta, \theta + \pi - \eta], \, f(z+re^{i\varphi}) \leq f(z) - 
\eta'.
\]
This is easily seen from writing a Taylor expansion around $z$,
choosing $\theta + \pi/2 - \eta/2$ to be the direction opposite to the gradient
and choosing $r,\eta'$ small enough. We fix such a pair
$r,\eta'$. Then we define a sequence of stopping time $\tau_k$ by
$\tau_0 = 0$ and  $\tau_{k+1} = \inf \{t > \tau_k | X_t \notin
B(X_{\tau_k},r) \}$. \cref{lem:eta} and the definition of $r,\eta'$
above show that until the random walk exits $R \setminus B_2$, it
satisfies $\P[f(X_{\tau_{k+1}}) \le f(X_{\tau_{k}}) - \eta'] \geq
\eta$. Call a step good if $f(X_{\tau_{k+1}}) \le f(X_{\tau_{k}}) - \eta'$. 
Since the value of $f$ on $\partial R$ is bigger than that in
$B_1$, if the first $O(1/\eta')$ steps are all good then the
random walk must hit $B_2$ before exiting $R$. Clearly this has positive probability so we are done.
\end{proof}

\subsection{Recurrence}

In this section, we prove that the random walk on a T-graph is recurrent. The 
main point is an asymptotic estimate of the conjugate Green function (i.e the 
discrete version of argument function). Let us note that the proof of this 
estimate comes from the exact solvability of the dimer model.

\begin{prop}\cite{KenyonHex}
  Let $T$ be a T-graph of parameters $\lambda$, $\Delta$. Let $w$ be a face of  
$T$ and let $d$ be a half line from the interior of $w$ to infinity that avoids 
all vertices of $T$. There exists a unique (up to a constant) function $G^*_{w 
d} : T \rightarrow \mathbb C$ that is discrete harmonic except for a  
discontinuity when crossing $d$: that is, if $q_{xy}$ are the rates of the continuous time simple random walk on $T$ then
$$
\sum_{y \sim x} q_{xy} (f(y) - f(x)) = \#\{y \sim x: (x,y) \text{ crosses $d$ clockwise } \} - \#\{y \sim x: (x,y) \text{ crosses $d$ anticlockwise } \}  
$$
This function 
satisfies
  \[  
  G^*_{wd}(\psi(b))  = \frac{1}{2\pi} \Bigl( \arg_d(\psi(b)-w) + \frac{ \Im \left(\lambda (\tfrac{\beta}{\gamma})^{m(w)} (\tfrac{\beta}{\alpha})^{n(w)} \right)}{\Re \left(\lambda (\tfrac{\beta}{\gamma})^{m(w)} (\tfrac{\beta}{\alpha})^{n(w)} \right)} \log \abs{\psi(b)- w} \Bigr) + C +O(1/|\psi(b)-w| )
\]
where $\arg_d$ denotes the determination of the argument with a $2\pi$ 
discontinuity on the half line $d$, $\alpha, \beta ,\gamma$ are the angles of 
$\Delta$ and $C$ is the constant up to which $G^*_{wd}$ is defined.
\end{prop}

In fact, note that because $T$ is a T-graph, at most one neighbour of $x$ is such that $(x,y)$ crosses $d$. Note that by definition, if $X_t$ is a random walk on $T$, then $G^*_{wd} (X_t) - \# (\text{signed crossings of $d$ before time $t$})$ is a martingale. From this it is easy to check recurrence.

\begin{thm}\label{T:recurrence}
The random walk on a T-graph is recurrent.
\end{thm}
\begin{proof}
Let $T$ be a T-graph and let $e$ be an (oriented) edge of $T$ to be chosen appropriately later. 
 Let $w^+$ and $w^-$ denote the triangles on the left and right of $e$. Let $X_t$ denote the random walk started at the end of $e$, let $V_t$ denote the number of jumps along $e$ before time $t$. For $R > 0$, let $\tau_R$ denote the exit time from the ball of radius $R$ centered at $e$. Let $d^+$ be a half line starting in $w^+$, crossing $e$ to go into $w^-$ and avoiding all vertices of $T$. Let $d^-$ be a piece of $d^+$ starting in $w^-$ and going to infinity. for ease of notation, we write $G^+ = G^*_{w^+ d^+}$, $G^- = G^*_{w^- d^-}$.

Note that we can always chose $e$ so that
\[
\frac{ \Im \left(\lambda (\tfrac{\beta}{\gamma})^{m(w^+)} (\tfrac{\beta}{\alpha})^{n(w^+)} \right)}{\Re \left(\lambda (\tfrac{\beta}{\gamma})^{m(w^+)} (\tfrac{\beta}{\alpha})^{n(w^+)} \right)} \neq \frac{ \Im \left(\lambda (\tfrac{\beta}{\gamma})^{m(w^-)} (\tfrac{\beta}{\alpha})^{n(w^-)} \right)}{\Re \left(\lambda (\tfrac{\beta}{\gamma})^{m(w^-)} (\tfrac{\beta}{\alpha})^{n(w^-)} \right)}.
\]
We note that $V_t - G^+(X_t) + G^-(X_t)$ is a martingale. Further, we have 
$G^+(X_{\tau_R}) - G^-(X_{\tau_R}) = C \log R + O(1)$ for some $C$. Therefore 
$\E V_{\tau_R} = C\log R + O(1)$ which implies that the probability of reaching 
distance $R$ before returning to $e$ decays like $O(1/\log R)$. This implies 
recurrence of the endpoint of $e$. Since a nondegenerate T-graph is irreducible (Lemma 3.23 in \cite{LaslierCLT}), the whole graph is recurrent.
\end{proof}

\section{Dimers and UST}\label{sub:bijection}

\subsection{Definition of the mapping}\label{sec:mapping}

In this section we describe the mapping from a forest on $T$ to a perfect 
matching of $\cH$ (a subset of edges of $\cH$ such that every vertex is incident to exactly one edge, i.e., a dimer configuration). This mapping was defined in 
\cite{dimer_tree} actually in a more general setting. We first give the 
construction in a full plane setting before addressing finite domains.

Let $T$ be a fixed T-graph and let $F$ be a spanning forest on $T$. In the context of an oriented graph as here, this means the following: every vertex has a single outgoing edge in $F$ and there is no loop in $F$, even ignoring the orientation. The forest $F$ is not necessarily connected and each connected component of $F$ is called a tree. Trees in $F$ naturally inherit the orientation from the T-graph $T$ and it is easy to see that each tree has a unique branch oriented toward infinity. Let $F^\dagger$ denote the dual of $F$, $F^\dagger$ is also a spanning forest with no finite connected component. For each connected component of $F^\dagger$ we choose one end and orient edges towards that end. Observe that if $F$ is a one-ended tree then $F^\dagger$ is also a one ended tree and there is therefore no arbitrary choice in the orientation of $F^\dagger$.

\begin{defn}\label{def:tree_matching}
Given $F^\dagger$ and its orientation we can construct a dimer configuration as follows. To each white vertex $w$ of $\cH$, is associated a face $\psi(w)$ of $T$. In $F^\dagger$, there is a unique outgoing edge starting in $\psi(w)$. This edge crosses an edge part of a segment of $T$ and by construction this segment corresponds to some black vertex $b \in \cH$ which is adjacent to $w$ in $\cH$ (see \cref{prop:planarity1}). We define this vertex $b$ to be the match of $w$. We denote by $M(F^\dagger)$ the matching constructed by the above procedure. By abuse of notation when $F$ is a one ended tree we write $M(F)$ for $M(F^\dagger)$ with its unique orientation.
\end{defn}

\begin{prop}\label{prop:is_matching}
Let $F$ be a spanning forest of a non-degenerate T-graph and let $F^\dagger$ denote an orientation of its dual. In the above construction, $M(F^\dagger)$ is a perfect matching of $\cH$.
\end{prop}
\begin{proof}
It is clear that every white vertex is matched by construction to a single 
black vertex so we only need to check that no black vertex is matched with two 
white vertices. Let $b$ be a black vertex and let $v \in T$ be the unique vertex 
in the interior of the segment $\psi(b)$ associated to $b$. There is exactly one 
outgoing edge from $v$ in $F$ so $\psi(b)$ is crossed by exactly one edge in 
$F^\dagger$ which of course has only one orientation. By construction this 
implies that $b$ is matched to only one white vertex.
\end{proof}


The main quantity of interest in a dimer configuration is often the height function, which completely encodes the dimer configuration. We will also explain below how it can be related to the winding of branches in the associated forest $F$. 
However, recall from \cref{sec:height_intro} that the definition of the height involves an arbitrary choice of a ``reference flow''; we therefore first describe how to choose this flow in order to have an exact identity. 

\begin{figure}
\begin{center}
\includegraphics[width=.9\textwidth]{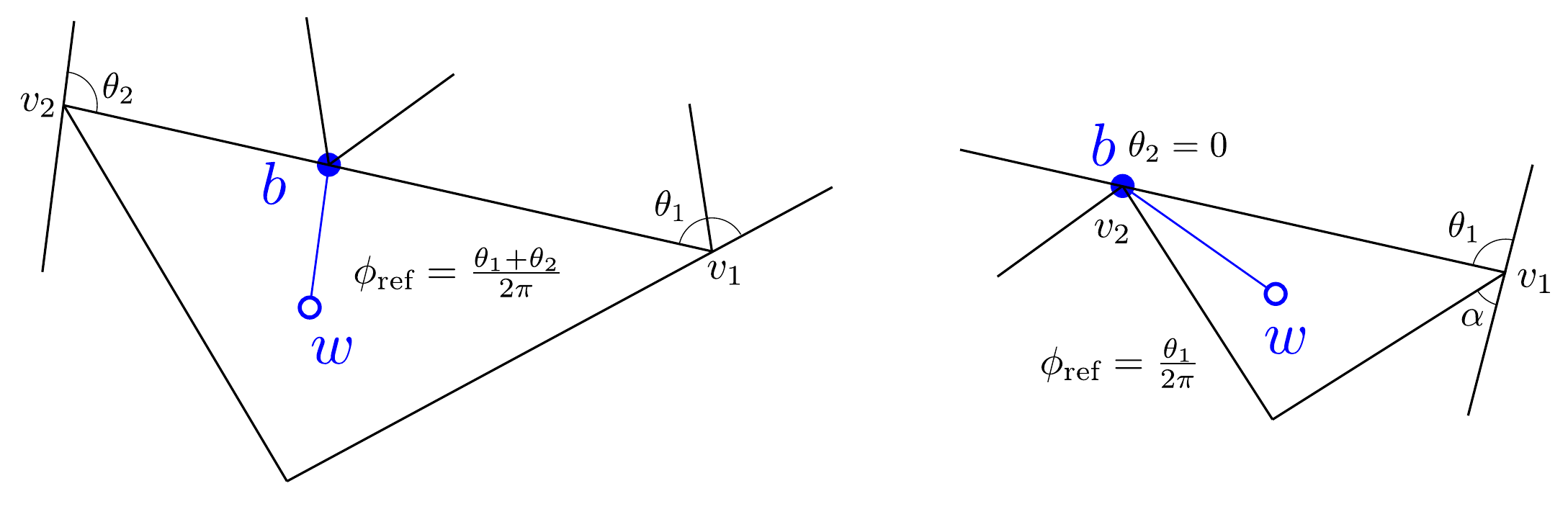}
\caption{The reference flow associated to a T-graph. The left hand side shows a case where $S_b$ is different from $S_1$ and $S_2$ while the right hand side shows a case with $S_2 = S_b$. Note that the local structures at $v_1$ and $v_2$ can be of two possible types as seen in the left hand side. In particular the angle $\alpha$ in the right hand side could be equal to $0$.}\label{fig:def_reference}
\end{center}
\end{figure}

The right choice of reference flow was introduced in \cite{KenyonHex} and is 
defined geometrically as in \cref{fig:def_reference}. Let $w$ and $b$ be 
adjacent white and black vertices. Let $v_1$ and $v_2$ be the two vertices of 
$T$ on both sides of $\psi(w) \cap \psi(b)$ and let $S_1$ and $S_2$ be the two 
segments containing $v_1$ and $v_2$ in their interior respectively (note that 
one of them might be $\psi(b)$). The flow from $w$ to $b$ is defined to be 
$\phi_{\text{ref}}(wb)=\frac{\theta_1 + \theta_2}{2\pi}$ where $\theta_1$ (resp. 
$\theta_2$) is the angle between $\psi(b)$ and $S_1$ (resp. $S_2$) measured 
opposite to $\psi(w)$ and without sign. 

\def\phiref{\phi_{\text{ref}}}

\begin{prop}\label{P:ref}
The reference flow satisfies for every white vertex $w \in \cH$, $\sum_{b \sim w} \phiref(wb) = 1$, and for every black vertex $b \in \cH$, $\sum_{w \sim b} \phiref(bw) = -1$.
\end{prop}

The proof of the proposition follows from \cref{fig:ref_flow} which illustrates the generic situation for a face in a T-graph, note that some of the angles however may be equal to zero. \cref{P:ref} says that the divergence of $\phiref$ is $\pm 1$ at every white/black vertex; as we will see in a moment, this property allows us to define a function for a dimer configuration unambiguously and hence we speak of a \emph{reference flow}.

\begin{figure}
\begin{center}
\includegraphics[width=.7\textwidth]{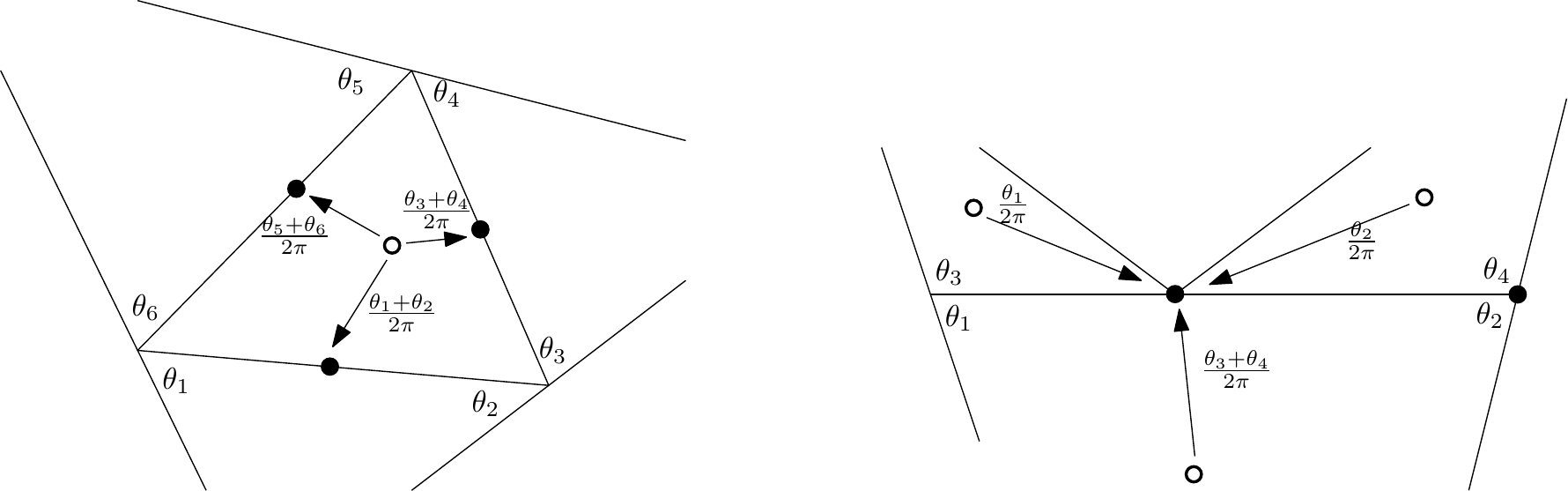}
\caption{An illustration of why $\phi_{\text{ref}}$ is a 
valid reference flow. In the left figure, some of the 
angles might be $0$.}\label{fig:ref_flow}
\end{center}
\end{figure}



Let $M$ be a dimer configuration on $\cH$ (seen as a subset of edges) and recall the definition of the height function: We define a flow $\phi_M$ by $\phi_M(wb) = - \phi_M( bw) = 1$ if $(wb) \in M$ and 0 otherwise. Then $\phi_M$ is also a flow and satisfies the same property as $\phiref $ in \cref{P:ref}. Hence $\phi_M - \phiref$ is a divergence free flow. Consequently there exists a unique (up to constant) function $h$ defined on the vertices of $\cH^\dagger$ such that if $e = bw$ is an edge of $\cH$ and $e^\dagger = xy$ is the corresponding dual oriented edge in $\cH^\dagger$ then
$$
(\phi_M  - \phiref)(wb) = h(y) - h(x).
$$
$h$ is called the \textbf{height function} of the (infinite) dimer configuration $M$.

\medskip The main result of this section relates the height differences to winding of branches when the dimer configuration derives from a one-ended spanning tree. Let $F$ be a spanning tree of $T$. Let $v,v'$ be two vertices of $\cH^\dagger$. Let $\gamma_{v\to v'}$ be the unique path in $F$ joining $v$ and $v'$. That is, $\gamma_{v \to v'} = (\gamma_0, \gamma_1, \ldots, \gamma_n)$ where $\gamma_0 = v$ and $\gamma_n = v'$ and the unoriented edge $\{ \gamma_i, \gamma_{i+1}\} \in T$. To this we add two vertices $\gamma_{-1}$ and $\gamma_{n+1}$ as follows: let $S_0$ (resp. $S_n$) be the segment of which $\gamma_0$ (resp $\gamma_n$) is the midpoint.
We require $(\gamma_{-1}, \gamma_0)$ to be perpendicular to $S_0$ and $\gamma_{-1}$ is on the side of $S_0$ containing the two incoming edges to $\gamma_0$. Likewise, we require $(\gamma_n, \gamma_{n+1})$ to be perpendicular to $S_n$ and $\gamma_{n+1}$ is on the side of $S_n$ not containing the incoming edges to $\gamma_n$. See \cref{fig:thetav} for an illustration.
  
\begin{thm}\label{lem:heightwinding}
Let $T$ be a non degenerate T-graph and let $F$ be a one ended spanning tree of 
$T$. Let $h$ denote the height function of $M(F)$ with the reference flow 
$\phi_{\text{ref}}$ defined above. We have for any $v, v' \in \cH^\dagger$,
\[
	h(v')  - h(v) =\frac{1}{2\pi} W_{\i} (\gamma_{v \to v'}) 
\]
where $\gamma_{v \to v'}$ is the path in $F$ joining $v$ and $v'$ as defined above, and
 $W_{\i}$ denotes the sum of the angles turned (with 
signs) from the initial point $\gamma_{-1}$ to final point $\gamma_{n+1}$ of the path.
\end{thm}
\begin{proof}
First note that with this definition $W_{\i}(\gamma_{v \to v'})$ is additive and antisymmetric (in both cases this is not completely obvious because of the beginning and end portions added to the path); see  \cref{fig:thetav}.
By additivity, we can restrict ourself to the case
where $\psi(v)$ and $\psi(v')$ are joined by a single edge in the
tree. Furthermore by antisymmetry we can assume without loss of 
generality that the edge in
the tree is oriented from $\psi(v)$ to $\psi(v')$ as in the left hand side of \cref{fig:thetav}.

\begin{figure}
\begin{center}
\includegraphics[width=.4\textwidth]{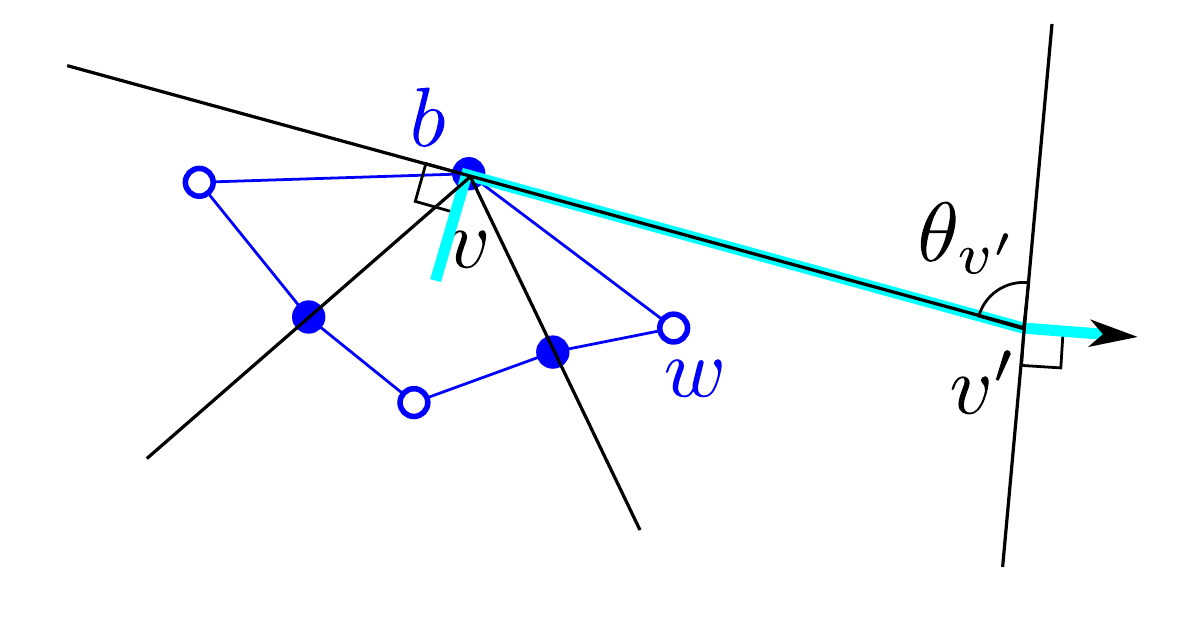}
\includegraphics[width=.4\textwidth]{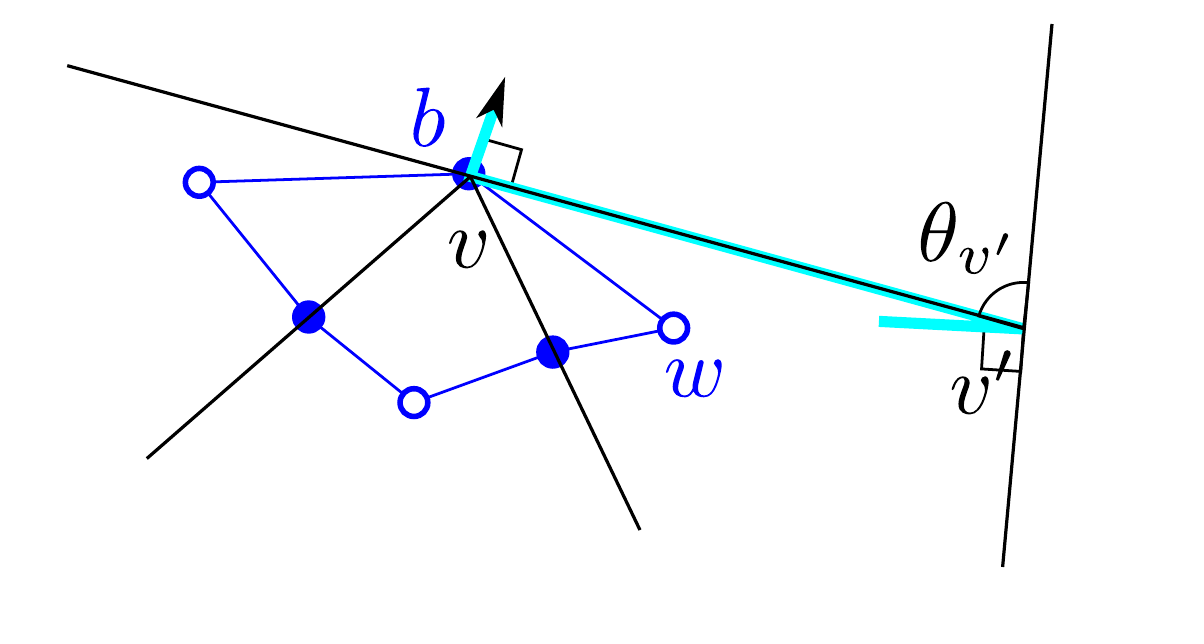}
\caption{Illustration of the proof of \cref{lem:heightwinding}. The hexagonal lattice is shown in blue. Left: The cyan path is $\gamma_{v \to v'}$ and in this case its intrinsic winding is clearly $-\frac{\pi}2-\theta_{v'}+\frac\pi 2 = -\theta_{v'}$. Right: The cyan path is $\gamma_{v' \to v}$ and the winding is $\theta_{v'}$.}
\label{fig:thetav}
\end{center}
\end{figure}

Let $b$ be the black vertex associated to $v$. Since there is an edge from 
$\psi(v)$ to $\psi(v')$ (which we have assumed is oriented from $\psi(v)$ to $\psi(v')$), it must be that $\psi(v')$ is an endpoint of $\psi(b)$. So $v$ and $v'$ 
are neighbours in $\cH^\dagger$ (this is obvious if we think of the segment as a flattened triangle, where the edge between the midpoint and the extremity is just one of the edges of the triangle). Let $w$ be the white vertex of the edge $e$ of 
$\cH$ which is dual to the edge between $v$ and $v'$. Note that the black vertex 
of $e$ is $b$. From the construction of the T-graph, we see easily that 
$\psi(w) \cap \psi(b)$ is the edge $[\psi(v), \psi(v')]$ of $T$. Since that edge 
is in $F$, the outgoing edge  from $\psi(w)$ in $F^\dagger$ does not cross 
$\psi(b)$. Hence $b$ and $w$ are not matched together and
\[
h(v') - h(v) =\pm  \phiref(wb)
%
\]
Recall that $\abs{ \phi_{\text{ref}}(wb) }= \frac{\theta_v + \theta_{v'}}{2\pi}$ where $\theta_v$ and $\theta_{v'}$ are angles at $v$ and $v'$. However by definition $b$ is the black vertex associated to $v$ so $\theta_v=0$. Also, it is immediate from \cref{fig:thetav} that $\abs{W_{\i} (\gamma_{v \to v'})} = \theta_{v'}$. It remains only to check that the signs are consistent with our conventions. For this, note that since the orientation of black triangles are preserved, if the tree edge $[\psi(v), \psi(v')]$ goes to the right (resp. left) then the edge $vv'$ is crossed with the white vertex to its right (resp. left). Therefore the sign of $\phi_{\text{ref}}(vv')$ is determined by this turn and it is easy to check that it is opposite to the sign of $W_{\i} (\gamma_{v \to v'})$.
\end{proof}

\begin{remark}
Note that we add two auxilary segments at the initial and final points of the paths (the segments $(\gamma_{-1},\gamma_0)$ and $(\gamma_n,\gamma_{n+1})$) which has a contribution to the winding which we did not include in the definition of Theorem 5.1 in \cite{BLR16}. However, we can add this to term $m^{\# \d}$ there since this extra winding is strictly a function of the point in the $T$-graph and not of the uniform spanning tree.
\end{remark}

\begin{remark}
There exists a similar construction of T-graphs for dimers on the square lattice 
and actually even any periodic bipartite planar graph and 
a version of \cref{lem:heightwinding} holds in that more general setting (see 
\cite{dimer_tree}). We do not work in this more general setup because 
the Central Limit Theorem of \cite{LaslierCLT}, which is needed for the convergence of the height function to the Gaussian free field in \cite{BLR16} is currently only proved for the hexagonal lattice (although the proof would extend to the square lattice without major difficulties).
\end{remark}

\begin{remark}\label{rem:O1}
The reference flow $\phi_{\text{ref}}$ is essentially equivalent to the flow $\phi_{p_a p_b p_c}$ defined as $\phi_{p_a p_b p_c}(wb) = p_a$ (resp. $p_b$, $p_c$) if $wb$ is a vertical (resp. NE-SW, NW-SE) edge. Indeed it follows from \cite{KenyonHex} Section 3.2.3 that if $h$ and $h'$ are the two heights functions associated to the same tiling but with references $\phi_{\text{ref}}$ or $\phi_{p_a p_b p_c}$ then $\abs{h - h'} \leq C$ for some universal constant $C$.
\end{remark}

%
%

\subsection{Dimer configurations on finite domains}\label{sec:domain_construction}
\def \inte {\text{int}}
\def \ext {\text{ext}}
 
In the previous section we explained how to associate to a spanning tree of an (infinite) T-graph a dimer configuration on the whole hexagonal lattice. We now explain how to extend this construction to spanning trees of finite subgraphs of a T-graph for a given finite domain of the hexagonal lattice. This will require choosing the boundary of the discrete domain on the T-graph in a careful way. In this section, we assume that we are given a continuous domain and that we want to approximate it by a``good" subgraph of the hexagonal lattice.

Let $U \subset \C$ be a bounded domain such that $\partial U$ is locally connected.  
Let $x \in \partial U$ be a marked point on the boundary. Consider an (infinite) T-graph corresponding to probabilities $p_a,p_b,p_c$ (chosen as above) and call it $T$. Let $\ell$ be the linear 
map as described in \cref{prop:prop_geometrique}, we recall that the map $\psi$ 
from $\cH^\dagger$ to $T$ is almost equal to $\ell$.

We now describe how to construct 
subdomains $U^\delta$ of the scaled lattice $\delta \cH$  together with a marked point $x^\delta$ which approximate $(U, x)$ in the sense that the boundary of $U^\d$ approximates $\partial U$ as closed sets and that they have a marked boundary face $x^\delta$ such that $x^\delta \to x$. They also satisfy the following properties.\ Firstly, when viewed as a lozenge tiling of a subdomain (and thus equivalently as a stack of cubes in $3$-d), the cubes on the boundary will be within $O(1)$ of some fixed plane $\cP$. Another way of putting it is that if $p_a, p_b, p_c$ are such that $(p_a, p_b, p_c)$ is normal to the plane $\cP$ then the height function of any dimer configuration on $U^\delta$, measured with respect to the flow $\phi_{p_a, p_b, p_c}$, will be within $O(1)$ from $\cP$ along the boundary.
Actually using \cref{lem:heightwinding}, the height differences with respect to $x^\delta$ along the boundary are also given by the winding of the boundary of $U^\delta$. Secondly, the dimer configuration will correspond to a wired UST by a finite analogue of the construction introduced in \cref{sec:mapping}.

 In what follows $o_\delta(1)$ denotes a function which goes to $0$ as $\delta \to 0$.


Write $\tilde U $ for $\ell (U)$ and fix a conformal map $\ph: \D \mapsto  \tilde U$.
We recall that $\ph$ extends continuously to $\bar \D$ if and only if $\partial \tilde U$ is locally connected, by Theorem 2.1 of \cite{Pommerenke}. We will assume that $U$ and thus $ \tilde U$ has locally connected boundary throughout this section. Hence $\partial U$ and $\partial \tilde U$ can each be seen as a curve\footnote{i.e. as a possibly non-injective continuous mapping $\gamma: \mathbb S_1 \mapsto \C$}. Consider  $ \tilde U_\ve:=  \ph((1-\ve)  \D)$ (note that $\tilde U_\eps \subset \tilde U$). Using the above information, we deduce that $\partial \tilde U_\ve \to \partial \tilde U$ in uniform norm between curves up to reparametrisation\footnote{In a non-locally connected boundary, like the topologist's comb, this convergence might fail.}.
\medskip

Now we claim that one can find a sequence of directed simple loops in the scaled T-graph $\delta T$ such that they converge to $\partial  \tilde U$ in the uniform norm up to reparametrisation. To see this, consider the image of $T^\d:= \delta T \cap \tilde U$ under $\ph^{-1}$. We will first find a path (in fact, an oriented loop) on $ \ph^{-1}(T^\delta) \subset \D$ such that the path will remain at distances between $1- 2\eps$ and $1-\eps/2$ from the origin.  
Indeed, we can approximate the circle $(1-\ve)\partial \D$ by a sequence of rectangles of dimension $\ve/10 \times 3\ve/10$ lying strictly inside $(1-\ve/2)\partial \D$ such that the starting ball of one rectangle is the target ball of another one. We can use the crossing assumption on $T$ and the bound on the derivative of $|\ph'|$ to say that these rectangles also satisfy the crossing assumption on the graph $\ph^{-1}(T^\delta)$ for small enough $\delta(\ve)$. Indeed, for each such rectangle $R$ with starting ball $B_1$ and target ball $B_2$, one can find a sequence of rectangles $R_1,\ldots, R_k \subset \ph(R)$ in $\tilde U$ each with aspect ratio $3:1$ such that the starting ball of $R_1$ contains $\ph(B_1)$ and the target ball of $R_k$ contains $\ph(B_2)$ and the starting ball of $R_i$ is the target ball of $R_{i-1}$ for $i=1,\ldots,k$. Observe that the size of $R_i$ can be chosen to be at least $c(\ve)$ where $c(\ve)>0 $ depends only on $\eps$ and $U$ (in fact, on the minimal value of $|\ph'|$ in $(1-\ve/4)\D $). Using the crossing estimate, for small enough $\delta=\delta(\ve)$, there exists at least one path crossing each rectangle $R_i, 1\le i \le k$. These paths can be concatenated. 
Consequently we obtain a loop in $T^\d$ which is formed by concatenation of paths lying inside the images of the rectangles in $\D$. We can assume that this path is a directed simple loop by loop-erasing and note that by construction, this still keeps some portion of the loop in $\ph(R)$ for each rectangle $R$ in $\D$. Call this loop $C^\d$. Note that $\ph^{-1}(C^\d)$ is a simple loop in $\D$ which lies between $(1-\ve/2)\partial \D$ and $(1-2\ve)\partial \D$ and hence $C^\d \to \partial \tilde U$ in the uniform norm when parametrised by the argument as $\ve \to 0$ since $\ph$ extends continuously to $\partial \D$. 

 Our second claim is that we can find a simple path $\tilde P^\d$ starting from a point on $C^\d$ within $o_\delta(1)$ distance of $\ell(x)$, and such that $\tilde P^\d$ goes off to infinity avoiding $C^\d$ except for the starting point. To see this the crucial point again is that we have a bound of the derivative of $\ph$ on a neighbourhood of $C^\delta$. 
 
 Let us give a detailed explanation. We move back to the unit disc and assume without loss of generality that $\ph^{-1}(x)$ is $1$. From the bound on the derivative, note that every point in $\ph(\D) = \tilde U$ outside $\tilde U_{1- \eps/100} = \ph((1-\ve/100)\D)$ is at least $c_\ve$ away from $C^\d$. Furthermore, consider the rectangle $R = [1-10\ve,1-\ve/200) \times [-\ve/200,\ve/200] \subset \D $. By planarity and the crossing estimate there exists a simple path $p \subset \ph(R)$ in the graph $T^\delta$, connecting a vertex in $C^\d \cap \ph(R)$ to a vertex  outside $\ph((1-\ve/100)\D)$.
 
   The curve $C^\delta$ is Jordan and therefore defines an inside $I^\delta$ and an outside. Call $y^\delta$ the end point of the path $p$ constructed above. Note that since $y^\delta \in \C \setminus I^\delta$ we can find a continuous path $\gamma$ which connects $y^\delta$ to infinity, avoiding $\bar I^\delta$. By compactness, note that $\gamma$ is at distance at least $c'_\eps$ from $\bar I^\delta$. Thus using a suitable sequence of rectangles of size $c'_\eps/2$ and using the crossing estimate, we can find an infinite oriented simple  path $p'$ in $T^\d$ which starts from $y^\d$ and avoids $C^\d$. We define $\tilde P^\d$ to be the portion of the path $p \cup p'$ from the last time it crosses $C^\d$. The point where $p \cup p'$ crosses $C^\d$ for the last time is within $o_\ve(1)$ of the starting point of $p$ by the argument in the previous paragraph and hence within $o_\delta(1)$ of $x$. 

\medskip
We root $C^\d$ at the vertex $x^\delta$. Recall that $C^\delta$ is a rooted oriented cycle. 
Let 
$P^\delta$ be the path obtained from $C^\d$ deleting the first edge of $C^\delta$ and concatenating with $\tilde P^\delta$. Thus $P^\delta$ is a simple oriented path going to infinity (which roughly speaking starts by looping around $\tilde U$ and then follows $\tilde P^\delta$). 
Call the erased edge 
  $e^\d$.  Since each vertex of the T-graph corresponds to a face of the hexagonal lattice $\cH^\dagger$, $\psi^{-1} 
(C^\d)$ is a self avoiding cycle in $\delta \cH^\dagger$ by \cref{prop:planarity2}. We define $\tilde U^\d$ 
to be the subgraph of $\delta \cH$ spanned by vertices strictly inside that 
cycle. The dual edge of $\psi^{-1}(e^\d)$ is an edge of $\delta \cH$ with 
exactly one of its endpoints in $\tilde U^\d$ and we define $U^\d$ by removing this 
vertex from $\tilde U^\d$. Define the boundary faces of $U^\d$ to be the faces 
in $\psi^{-1}(C^\d)$.
(Note that it may happen that $x^\delta$, although a boundary face of $U^\delta$ in this definition, is not actually adjacent to $U^\delta$ after the vertex removal, but is always within $O(\delta)$ from $U^\delta$, see \cref{fig:domain}) 

\medskip 
Now the idea will be to use the connection between dimers and trees in the whole plane and to that end we will extend $P^\d$ to a one ended tree in the full plane.

Let 
$D^\delta$ be the sub-graph of $\delta T$ consisting of all the vertices 
which are strictly inside $C^\delta$ and all the outgoing edges from 
them (including their endpoints). Clearly, every vertex in $D^\delta$ has outdegree $0$ or $2$. Again, 
using the crossing estimate, $P^\delta$ can be extended to a one-ended spanning 
tree of $T^\delta$. Furthermore the extensions inside and outside $C^\d$ do not 
depend on each other (in the sense that given an extension 
outside, every extension inside is possible) because the edge $e^\delta$ is never used. The possible 
extensions in $D^\d$ are exactly the wired trees of $D^\delta$ (where we glue 
together all the degree $0$ vertices). 

Let $F$ be an arbitrary extension of $P^\d$ into a one-ended tree in $\delta T$. Notice that $F^\dagger$ can be decomposed into 
$F^\dagger = F_\inte^\dagger \sqcup F_\ext^\dagger \sqcup (e^\d)^\dagger$, 
where 
$F_\inte^\dagger$ is the portion of $F^\dagger$ spanned by vertices lying 
completely inside $C^\d$, $F_\ext^\dagger$ is the portion of $F^\dagger$ 
spanned by vertices lying strictly outside $C^\d$ and $(e^\d)^\dagger$ is the 
dual edge to $e^\delta$. Furthermore $F_\ext^\dagger$ and $F_\inte^\dagger$ are 
joined together by $(e^\d)^\dagger$ oriented from $F_\inte^\dagger$ to 
$F_\ext^\dagger$. Hence the oriented tree $F_\inte^\dagger$ is a function of 
the oriented tree $F_\inte$.


\medskip Recall the map $M$ 
from \cref{def:tree_matching} which produces a matching from a one-ended spanning tree.
$M(F^\dagger)$ is a perfect matching of the whole hexagonal lattice (by 
\cref{prop:is_matching}). Let $h$ be the associated height function as in \cref{lem:heightwinding}, defined up to a global additive constant. Note that $F^\dagger$ contains no edge connecting the inside of $C^\delta$ with the outside of $C^\delta$ (except for $(e^\d)^\dagger$).
 %
 Therefore $M(F_\inte^\dagger)$ is a perfect 
matching 
of $U^\d$. It then follows from 
Theorem 3.4 in \cite{dimer_tree} (applied to $F^\dagger_\inte$) 
that the image of the wired UST measure on $D^\d$ is 
the 
uniform dimer measure on $U^\d$. 
Furthermore, for any $v \in \partial 
U^\d$, $h(v) - h(x^\d)$ is the intrinsic winding of the oriented path from 
$\psi(x^\d)$ to $\psi(v)$ in $C^\d$ (in the sense of 
\cref{lem:heightwinding}). Let us summarise our findings in the following 
proposition.


\medskip
\medskip

\begin{thm}\label{prop:domain}
The objects $U^\d, D^\d, x^\d,e^\d$ constructed above satisfy the following:
\begin{itemize}
\item The graph $U^\d$ is a subgraph of $\delta H$. Furthermore, the boundary faces of $U^\d$ (faces which have neighbour outside) can be joined together in the dual graph to form a simple loop and this loop approximates $U$ in 
 the uniform norm as unparametrised curved with $x^\d \to x \in \partial U$. Further $D^\d$ is a 
subgraph of $T^\d$.
\item The images by the construction of \cref{def:tree_matching} of the wired 
UST in $D^\d$ with the dual tree oriented towards $e^\d,$ are 
uniform dimer configurations $U^\d$.
\item For any $v \in \partial 
U^\d$, $h(v) - h(x^\d)$ is the intrinsic winding of the boundary of $D^\d$ 
between 
$\psi(v)$ and $\psi(x^\d)$ (more precisely, it is the winding of $P^\d$ 
constructed above between 
$\psi(v)$ and $\psi(x^\d)$ in the sense of 
\cref{lem:heightwinding}).
\end{itemize}
\end{thm}


\begin{figure}[h]
\includegraphics[scale = 1]{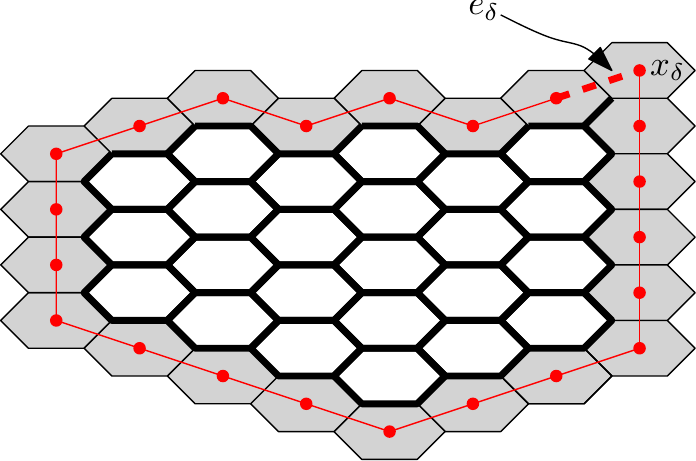}\hspace {1 cm}
\includegraphics[scale = 1]{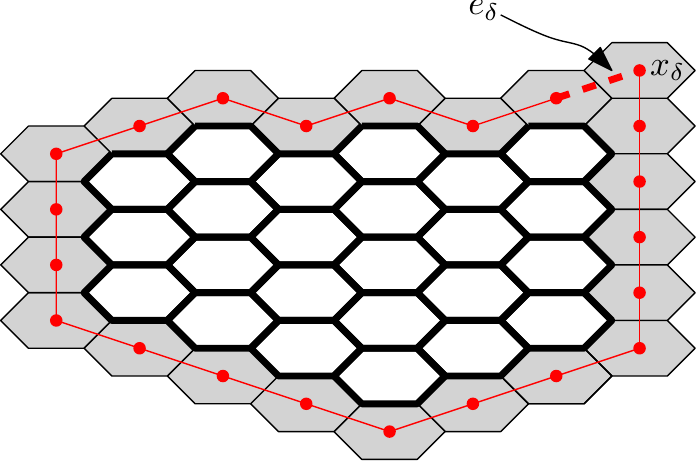}
\caption{Illustration of the construction of the domain $U^\delta$. The red loop (corresponding to the grey hexagons) is $\psi^{-1}(C^\d)$. The dotted red edge is $e^\delta$. On the left, the bold hexagons form $\tilde U^\d$. On the right we delete a vertex from $\tilde U^\d$ to obtain $U^\d$. In this case, $x^\delta$ is not actually adjacent to $U^\delta$.}\label{fig:domain}
\end{figure}

\begin{remark}
Given our choice of reference flow, the third point in \cref{prop:domain} 
together with \cref{rem:O1} implies that the height function on $U^\delta$, viewed as a surface in $2+1$ dimensions via cube stacks, lies within $O(1)$ of the plane $\cP$. The domains of 
 $U^\d$ constructed here can be thought of as a natural generalisation of \textbf{Temperleyan 
domains} in the case of arbitrary slope and to the hexagonal lattice.

Note that there exist subsets of the hexagonal 
lattice that admit dimer configurations but are not Temperleyan in the above sense.
\end{remark}

\subsection{Local limits}\label{sec:local}

In this section we show that the mapping $M$ between a full plane uniform spanning tree 
measure in a \textit{translation invariant version} of the T-graph corresponding to a certain slope produces 
a translation invariant ergodic Gibbs measure on dimer configurations with the same slope. The translation invariant version amounts to randomising the T-graph by picking a $\lambda $ which is uniform on the unit circle as explained in \cref{prop:translation}.
To emphasise a bit more, it is straightforward to see from locality of the map 
$M$ that once we establish a local limit for the trees the corresponding dimer 
configuration is Gibbs. Also since the T-graph is translation invariant, so is the dimer measure. However, translation invariant Gibbs measures are not unique and we need an 
extra argument to identify the limit. 

Note that it is not obvious from classical electrical network argument that the 
local limit of wired UST exists since our graph is oriented. However we have
established this in Corollary 4.20 in \cite{BLR16}, using the uniform crossing estimate of \cref{prop:satisfy_crossing}.

\begin{thm}\label{T:gibbs}
Fix $p_a, p_b, p_c >0$ be such that $p_a + p_b + p_c =1$.  Let $\lambda$ be picked uniformly at random from the unit circle. Let $T= T_{\lambda \Delta}$ be a (random) 
$T$-graph associated with $\lambda $  and $p_a,p_b,p_c$ as in \cref{prop:translation}. Let $\mu_F$ denote the uniform 
spanning tree measure on $T$ defined by taking a local limit of wired domains as 
in Corollary 4.20 in \cite{BLR16}. Then the random dimer 
configuration $M(\mu_F)$ (as in \cref{def:tree_matching}) is the unique translation invariant ergodic 
Gibbs measure on full plane lozenge tiling with probabilities $p_a, p_b, p_c$. 
\end{thm}

\begin{proof}
First, condition on $\lambda$ and take a sequence of domains as described in \cref{prop:domain} 
which exhaust the plane. Thus there is a 
measure preserving transformation between uniform spanning trees in this 
sequence and the uniform dimer configurations in these domains (i.e. with 
weight $1$ for each dimer). Since the transformation going from tree to a dimer 
configuration is local, we take limits of the uniform spanning trees via 
Corollary 4.2 in \cite{BLR16} and obtain that $M(\mu_F)$ is a 
Gibbs measure on lozenge tilings of the full plane. We now average over $\lambda$ to obtain a translation invariant Gibbs measure on dimer configurations which we write $\mu_{\text{dim}}$. We 
can decompose $\mu_{\text{dim}}$ into its ergodic parts which are fully characterised by 
\cref{thm:uniqueness_dimers}:
 $$
 \mu_{\text{dim}} = \int \mu_s d \nu(s)
 $$ for some measure 
$\nu$ on $S=\{s=(p_a, p_b, p_c)  \in [0,\infty)^3:  p_a + p_b + p_c =1 \}$.

Let $s_0= (p_a,p_b,p_c)$ denote the triplet used to construct the T-graph. 
Recall that if we measure height functions using the reference flow defined 
from the T-graph (see \cref{rem:O1}), we have  $\E_{\mu_{s_0}}[h(x)] = 
O(1)$ for all $x \in \cH^\dagger$ and $h(0) = 0$ (in the full plane we 
always need to pin the height at one point). For other measures it is also clear by \cref{thm:uniqueness_dimers} 
that $$\E_{\mu_s}[h(x)] - \E_{\mu_{s_0}} [h(x)] = l_s(x)$$ where $l_s$ is a 
non-zero linear function. Furthermore it is known that for any $s \in S$, the fluctuations 
around the mean of the height is at most $L^{1/4}$ with very high probability  
(in fact the right order is logarithmic, but not essential for this proof). 
For instance, it follows from \cite{TL14}, Theorem 2.8 for all $L >0$, for 
all $x \in B(0,L) \cap \cH^\dagger$,

\begin{equation}
 \P_{\mu_s}(|h(x) - \E_{\mu_{s}}(h(x))| > L^{1/4} ) <L^{-10} \label{eq:cont1}.
\end{equation}

%

Let us assume by contradiction the measure $\nu$ is not a Dirac mass at $s_0$. 
Consider the domain $U$ approximating $B(0,L )$ as in the 
construction of \cref{prop:domain} for a large radius $L$. On 
one hand, Theorem 2.9 in \cite{TL14} shows that under the uniform measure on 
$U$, 
\begin{equation}
 \P\Big(\sup_{x \in B(0,L)} \abs{h(x)} > L^{1/4}\Big) \leq L^{-10},
\label{eq:cont2}
\end{equation}
where the height $h(x)$ is measured with respect to the reference flow on the T-graph.  Since $\nu$ is not a Dirac mass at $s_0$, we can find some $\varepsilon >0$ such 
that the $\nu$ measure of $$\Sigma:= \{ s : \exists x, l_s(x) \ge 2\varepsilon 
|x|\}  $$ is at least $ 4 \varepsilon$. By \eqref{eq:cont1}, for any 
$L$ large enough we have 
$$
 \forall s \in \Sigma, \quad  \P_{\mu_s} \left( \sup_{x \in B(0, \sqrt L)} 
\abs{h(x)}\geq \varepsilon \sqrt L \right) \geq 1-\varepsilon ,
$$
where $h$ is still measured with respect to the reference flow. 
Along with the bound on $\nu(\Sigma)$, this implies
$$
\P_{\mu_{\text{dim}}} \left( \sup_{x \in B(0, \sqrt L)} \abs{h(x)}\geq \varepsilon \sqrt L \right) \geq 3 \ve.
$$
Now by the coupling of Corollary 4.2 in \cite{BLR16}, we can couple the dimer 
configuration in $B(0, \sqrt L)$ under the uniform measure $\mu_U$ in $U$ and 
the whole plane measure $\mu_{\text{dim}}$ with probability at least $1-\varepsilon$ (actually this was shown even conditioned on $\lambda$). In 
particular
$$
\P_{\mu_U} (\sup_{x \in B(0, \sqrt L)} \abs{h(x)}\geq \varepsilon \sqrt L) \ge 
\P_{\mu_{\text{dim}}} \Big(\sup_{x \in B(0, \sqrt L)} \abs{h(x)}\geq \varepsilon \sqrt L\Big) 
-\varepsilon \ge 2\varepsilon.
$$
This is a contradiction with \eqref{eq:cont2}.
\end{proof}

\bibliographystyle{abbrv}

\bibliography{winding.bib}

\end{document}